\documentclass[11pt]{article}

\usepackage{amsmath, amssymb, amsthm}
\usepackage[utf8]{inputenc}
\usepackage[T1]{fontenc}
\usepackage{geometry}
\usepackage{hyperref}
\usepackage{array}
\usepackage{booktabs}
\usepackage{listings}
\usepackage{xcolor}

\definecolor{codegray}{rgb}{0.5,0.5,0.5}
\definecolor{codegreen}{rgb}{0,0.5,0}
\definecolor{codepurple}{rgb}{0.58,0,0.82}
\definecolor{codeblue}{rgb}{0.13,0.13,0.55}
\definecolor{backcolour}{rgb}{0.97,0.97,0.97}

\lstdefinestyle{pythonstyle}{
    backgroundcolor=\color{backcolour},
    commentstyle=\color{codegreen}\itshape,
    keywordstyle=\color{codeblue}\bfseries,
    stringstyle=\color{codepurple},
    numberstyle=\tiny\color{codegray},
    basicstyle=\ttfamily\footnotesize,
    breakatwhitespace=false,
    breaklines=true,
    captionpos=b,
    keepspaces=true,
    numbers=left,
    numbersep=5pt,
    showspaces=false,
    showstringspaces=false,
    showtabs=false,
    tabsize=4,
    frame=single,
    framerule=0.4pt,
    rulecolor=\color{codegray},
    language=Python,
    xleftmargin=15pt,
    xrightmargin=5pt,
}
\newtheorem{theorem}{Theorem}[section]
\newtheorem{lemma}[theorem]{Lemma}
\newtheorem{proposition}[theorem]{Proposition}
\newtheorem{corollary}[theorem]{Corollary}
\theoremstyle{definition}
\newtheorem{definition}[theorem]{Definition}

\theoremstyle{remark}
\newtheorem{remark}[theorem]{Remark}

\newcommand{\C}{\mathbb{C}}
\newcommand{\R}{\mathbb{R}}
\newcommand{\D}{\mathbb{D}}
\newcommand{\HH}{\mathbb{H}}
\newcommand{\Az}{\mathcal{A}_z}
\newcommand{\dbar}{\bar{\partial}}

\newcommand{\ii}{\mathfrak{i}}
\newcommand{\Acal}{\mathcal{A}}
\newcommand{\Bcal}{\mathcal{B}}
\newcommand{\Fcal}{\mathcal{F}}
\newcommand{\Mcal}{\mathcal{M}}
\newcommand{\BV}{\mathcal{BV}}
\newcommand{\Gcal}{\mathcal{G}}

\title{The Pseudo-Analytic Mass of a Beltrami--Vekua Equation}
\author{Daniel Alay\'on-Solarz\thanks{\texttt{danieldaniel@gmail.com}}}
\date{May 2026}

\begin{document}

\maketitle

\begin{abstract}
Every smooth first-order real planar elliptic system admits a universal complex form $w_{\bar z} - \mu w_z + \mathcal{A} w + \mathcal{B} \bar w = \mathcal{F}$, which we call the \emph{Beltrami--Vekua equation}: the data $(\mu, \mathcal{A}, \mathcal{B}, \mathcal{F})$ are produced from the original system by algebraic operations and differentiations, with no auxiliary PDE. On this space we study the joint action of multiplicative gauges $w \mapsto \phi w$ and orientation-preserving diffeomorphisms. Our main result is that the 2-form $\Theta = |\mathcal{B}|^2 / (1 - |\mu|^2) \, dx \, dy$ is gauge-invariant and pulls back covariantly under diffeomorphisms; its form is forced, with $|\mathcal{B}|^2$ the unique $\mathcal{B}$-quadratic combination invariant under $\mathcal{B} \mapsto \mathcal{B}\phi/\bar\phi$ and $1 - |\mu|^2$ the conformal distortion factor from the diffeomorphism law for $\mu$. The total mass $\mathcal{M}(D) = \int_\Omega \Theta$, the \emph{pseudo-analytic mass}, vanishes precisely on the analytic class $\mathcal{B} \equiv 0$ and separates a continuous family of pairwise inequivalent pseudo-analytic equations on the disk. As a by-product, Vekua's two-stage reduction --- uniformization then gauge elimination --- requires only one variable-coefficient PDE solve: the Beltrami diffeomorphism supplies the integrating factor for a flat $\bar\partial$-equation.
\end{abstract}

\section{Introduction}\label{sec:intro}

A first-order real elliptic system on a planar domain $\Omega \subset \R^2$ has --- after the standard reduction to skeleton form --- eight real coefficients governing two real unknowns $u, v$. Such a system can always be rewritten as a single complex equation in one complex unknown $w = u + iv$, but the rewriting is not unique: different bundling strategies produce structurally different complex forms, each making different trade-offs between algebraic cleanliness and the regularity demanded of the data. Three such forms occupy distinct positions on this spectrum.

\medskip
\noindent
\textbf{Vekua's form} $w_{\bar z} + Aw + B\bar w = F$ \cite{vekua, bers} is reached by first solving a Beltrami equation to uniformize the principal part, then reading off the lower-order data in the new coordinates. The principal part is consumed in the first stage; the resulting form is algebraically clean --- one principal-part term, $\bar w$ in the conjugate slot --- and admits the rich function theory of generalized analytic / pseudoanalytic functions. The price is the PDE solve, which forces smoothness on the principal-part coefficients (Vekua works in $D_{m+1,p}$ on the principal part, $D_{m,p}$ on lower-order data; see \cite[\S 7.1]{vekua}).

\medskip
\noindent
\textbf{Bojarski's form} $w_z - q_1 w_{\bar z} - q_2 \overline{w_{\bar z}} = Aw + B\bar w + C$ \cite{bojarski} is reached by an algebraic elimination: complexify the skeleton, take the conjugate equation, and use the pair to eliminate $w_z$ without solving any PDE. The construction descends to measurable coefficients $q_1, q_2$ with $|q_1| + |q_2| < 1$ --- the natural setting for quasiconformal mappings with merely measurable Beltrami data. The price is that the principal part now carries two coefficients, and the conjugate slot holds $\overline{w_{\bar z}}$ rather than $\bar w$.

\medskip
\noindent
\textbf{The Beltrami--Vekua form} taken up in the present paper,
\begin{equation}\label{eq:BV-intro}
w_{\bar z} - \mu\, w_z + \Acal\, w + \Bcal\, \bar w = \Fcal, \qquad |\mu| < 1,
\end{equation}
sits between these two. Like Bojarski's form it is \emph{pre-uniformized}: the data $(\mu, \Acal, \Bcal, \Fcal)$ are constructed from the eight coefficients and two forcings of the original system by a finite sequence of algebraic operations and differentiations (Section~\ref{sec:BV-derivation}), with no PDE solved at any stage. Like Vekua's form it has a single principal-part term and $\bar w$ --- not $\overline{w_{\bar z}}$ --- in the conjugate slot. The cost of having both features at once is a single $C^1$ derivative on the principal-part data, paid at Step~2 when the structural obstruction $G$ is computed; the lower-order data and forcing need only be $C^0$.

The thesis of this paper is that the BV form combines Bojarski's epistemic stance --- keep the principal part inside the equation and study it jointly with the lower-order data --- with Vekua's algebraic form, and that this combination unlocks an invariant theory that neither classical form supports.

\medskip
\noindent
\textbf{Two natural equivalence relations.} The BV form \eqref{eq:BV-intro} is universal but not canonical: the data $(\mu, \Acal, \Bcal, \Fcal)$ depend on the gauge of the unknown $w$ and on the choice of coordinates on the domain. Two equations related by a multiplicative substitution $w \mapsto \phi w$ with $\phi$ a nowhere-vanishing $C^1$ function, or by an orientation-preserving $C^1$ diffeomorphism of the domain, describe the same underlying object. Sections~\ref{sec:gauge} and~\ref{sec:diffeo} establish that gauge transformations and diffeomorphisms each define equivalence relations on the space of BV equations and compute their action on the data.

\medskip
\noindent
\textbf{Main result: the pseudo-analytic mass.} The principal contribution of this paper is the identification of a natural gauge--diffeomorphism invariant density on the pre-uniformized BV data. Section~\ref{sec:dichotomy} first extracts an immediate Boolean consequence of the transformation laws: the locus $\{\Bcal = 0\}$ is invariant under both actions, separating an \emph{analytic} class ($\Bcal \equiv 0$) from a \emph{pseudo-analytic} class ($\Bcal \not\equiv 0$), in the terminology of Vekua and Bers. Section~\ref{sec:invariant} promotes this Boolean invariant to a continuous one. The 2-form
\begin{equation}\label{eq:Theta-intro}
\Theta := \frac{|\Bcal|^2}{1 - |\mu|^2}\, dx\, dy
\end{equation}
is gauge-invariant and pulls back covariantly under orientation-preserving diffeomorphisms (Theorem~\ref{thm:main-invariance}). The form of \eqref{eq:Theta-intro} is forced, not chosen: the numerator $|\Bcal|^2$ is the unique $\Bcal$-quadratic combination invariant under the phase-only gauge action $\Bcal \mapsto \Bcal\,\phi/\bar\phi$, and the denominator $1 - |\mu|^2$ is exactly the conformal distortion factor produced by the diffeomorphism law for $\mu$ via the closure identity $1 - |\mu^*|^2 = (1 - |\mu \circ \Phi|^2)\,J/|K|^2$ (Section~\ref{sec:diffeo}). The cancellation between these two transformation factors --- one from the conjugate Vekua coupling, one from the Beltrami principal part --- is the mathematical hinge of the paper. Its total mass
\begin{equation}\label{eq:mass-intro}
\Mcal(D) := \int_\Omega \Theta \;\in\; [0, \infty]
\end{equation}
is therefore an invariant of the gauge-diffeomorphism equivalence class of $D = (\mu, \Acal, \Bcal, \Fcal)$. We call $\Mcal(D)$ the \emph{pseudo-analytic mass}. It vanishes exactly on the analytic class and is strictly positive on the pseudo-analytic class; in particular, two equations with different mass cannot be equivalent. Section~\ref{sec:examples} exhibits a continuous one-parameter family of pairwise inequivalent pseudo-analytic equations on the disk distinguished by their mass, and shows that equal mass does not imply equivalence (the zero locus of $\Bcal$ is a separate invariant).

\medskip
\noindent
\textbf{What the pre-uniformized stance buys.} The invariant $\Theta$ is intrinsic to the pre-uniformized object: it is visible directly from $(\mu, \Bcal)$ without ever solving the Beltrami equation. After uniformization $\mu \mapsto 0$, the denominator in \eqref{eq:Theta-intro} becomes trivial and the invariant collapses to $|\tilde\Bcal|^2\,d\xi\,d\eta$ in the new coordinates --- where it is no longer recognizable as the pullback of a more refined object on the pre-uniformized domain. The mass is therefore not an artifact of the BV bundling but a genuine feature of the equivalence class that the classical Vekua reduction silently discards. This is what the pre-uniformized stance buys: invariants that the uniformizing solve renders invisible.

\medskip
\noindent
\textbf{The classical reduction inside BV.} Equipped with both equivalence relations, the classical two-stage reduction --- (i) solve the Beltrami equation to uniformize, (ii) gauge away the resulting $w$-coupling --- becomes a single composite operation. Section~\ref{sec:reduction} carries it out and notes a structural reorganization of the classical pipeline: the diffeomorphism $\Phi$ that solves $\Phi_{\bar z} = \mu \Phi_z$ already supplies the integrating factor for the gauge step. What appears in the classical literature as two independent variable-coefficient PDE problems is, in the BV framework, one --- the second stage reduces, after pullback through $\Phi$, to a flat $\bar\partial$-equation solvable by the standard Cauchy transform (Proposition~\ref{prop:single-solve}). The classical normal form $(0, 0, \tilde\Bcal, \tilde\Fcal)$ falls out as a corollary. This observation is conceptual rather than technical; the deeper contribution remains the invariant $\Theta$.

\medskip
\noindent
\textbf{Why $\Fcal$ does not yield an analogous mass density.} A natural question is whether the forcing term $\Fcal$ supports a similar construction. Lemma~\ref{lem:no-F-invariant} shows that no density of the displayed algebraic form built from $|\Fcal|^2$ can be gauge-invariant: under gauge, $\Fcal$ is multiplied by an unconstrained scalar $\phi$, while $\Bcal$ is multiplied by the unimodular factor $\phi/\bar\phi$. The asymmetry in the gauge action --- multiplicative-by-arbitrary-scalar versus multiplicative-by-phase --- is the structural reason $|\Bcal|$ admits an intrinsic 2-form while $|\Fcal|$ does not at this algebraic level. (The zero locus of $\Fcal$ is, however, preserved by both actions, just as for $\Bcal$.)

\medskip
\noindent
\textbf{Regularity.} The pipeline of Section~\ref{sec:BV-derivation} differentiates the principal-part coefficients of the original real system once, so we work throughout in the smooth class of Definition~\ref{def:smooth-class}: $C^1$ on the principal part, $C^0$ on lower-order data and forcing. This is the natural regularity for any diffeomorphism-based theory in this area --- Vekua's classical $D_{m+1,p}$ / $D_{m,p}$ split is the Sobolev counterpart, and the $C^1$ requirement on the principal part is its classical subset. The algebraic portions of the present construction (gauge laws, fiber-Vekua matching, Cayley conversion) admit natural Sobolev analogues with pointwise identities interpreted a.e.; a full Sobolev formulation of the diffeomorphism action, the Cauchy-transform step, and the associated invariant theory requires additional functional-analytic bookkeeping and is not pursued here. Bojarski's complementary route to measurable coefficients eliminates one Wirtinger derivative algebraically rather than by a coordinate change \cite{bojarski}, at the cost of the doubled principal part discussed above; the bundle map between his form and the BV form exists --- both descend from the same skeleton --- but the gauge-diffeomorphism invariant theory developed here does not transport, since the action $\Bcal \mapsto \Bcal\,\phi/\bar\phi$ that produces pointwise gauge-invariance of $|\Bcal|$ requires the $\bar w$ coupling rather than $\overline{w_{\bar z}}$.

\medskip
\noindent
\textbf{Outline.} Section~\ref{sec:real-system} fixes the real system and its ellipticity condition. Section~\ref{sec:BV-derivation} derives the BV form through a seven-step pipeline. Sections~\ref{sec:gauge} and~\ref{sec:diffeo} establish the gauge and diffeomorphism actions. Section~\ref{sec:reduction} carries out the classical reduction inside the BV framework and proves the single-PDE-solve proposition. Section~\ref{sec:dichotomy} extracts the analytic / pseudo-analytic dichotomy. Section~\ref{sec:invariant} constructs $\Theta$, proves the main invariance theorem, and rules out an $\Fcal$-analog at the algebraic level. Section~\ref{sec:mass} develops the pseudo-analytic mass. Section~\ref{sec:examples} carries out explicit computations. 

\section{The real system and the smooth class}\label{sec:real-system}

We consider two real-valued unknowns $u(x,y)$ and $v(x,y)$ on a domain $\Omega \subset \R^2$ with complex coordinate $z = x + iy$, governed by
\begin{equation}\label{eq:real-system}
\begin{cases}
-v_y + a_{11}\, u_x + a_{12}\, u_y + a_{13}\, u + a_{14}\, v = f_1, \\[4pt]
\phantom{-}v_x + a_{21}\, u_x + a_{22}\, u_y + a_{23}\, u + a_{24}\, v = f_2,
\end{cases}
\end{equation}
where the eight coefficients $a_{ij}(x,y)$ and the two forcing terms $f_k(x,y)$ are real-valued. The principal-part coefficients $a_{i1}, a_{i2}$ are assumed $C^1$ (they will be differentiated in Section~\ref{sec:BV-derivation}); the lower-order $a_{i3}, a_{i4}$ and $f_k$ need only be $C^0$.

The principal-part symbol is
\[
\sigma(\xi) = \begin{pmatrix} a_{11}\,\xi_1 + a_{12}\,\xi_2 & -\xi_2 \\ a_{21}\,\xi_1 + a_{22}\,\xi_2 & \xi_1 \end{pmatrix},
\]
with determinant $a_{11}\xi_1^2 + (a_{12}+a_{21})\xi_1\xi_2 + a_{22}\xi_2^2$. \emph{Ellipticity} means $\det\sigma(\xi) \neq 0$ for all real $\xi \neq 0$, equivalently
\begin{equation}\label{eq:ellipticity}
a_{11} > 0, \qquad a_{11} a_{22} - \tfrac{1}{4}(a_{12}+a_{21})^2 > 0,
\end{equation}
which forces $a_{22} > 0$. We assume \eqref{eq:ellipticity} holds throughout $\Omega$.

\begin{definition}[Smooth first-order planar elliptic system]\label{def:smooth-class}
A \emph{smooth first-order planar elliptic system} on $\Omega$ is a system of the form \eqref{eq:real-system} satisfying the ellipticity condition \eqref{eq:ellipticity}, with principal-part coefficients $a_{i1}, a_{i2} \in C^1(\Omega)$ and lower-order coefficients $a_{i3}, a_{i4}$ and forcing terms $f_k$ in $C^0(\Omega)$.
\end{definition}

This is the natural regularity class for the moving-algebra construction of Section~\ref{sec:BV-derivation} and for the gauge--diffeomorphism actions of Sections~\ref{sec:gauge}--\ref{sec:diffeo}: the principal part is differentiated exactly once (in computing the obstruction $G$), and the lower-order data enters only algebraically. The algebraic step (the seven-step pipeline of Section~\ref{sec:BV-derivation}) has a natural Sobolev counterpart in the sense of Vekua \cite[\S 7.1]{vekua} ($D_{m+1,p}$ on the principal part, $D_{m,p}$ on lower-order data), with pointwise identities replaced by their a.e.\ counterparts; the analytic apparatus of Sections~\ref{sec:diffeo}--\ref{sec:reduction} then requires the standard quasiconformal and Cauchy-transform machinery in those function spaces. We do not pursue this extension here.

\begin{remark}
The skeleton $(-v_y, v_x)$ is chosen so that the standard Cauchy--Riemann system $u_x - v_y = 0$, $v_x + u_y = 0$ fits \eqref{eq:real-system} with $a_{11} = a_{22} = 1$, $a_{12} = a_{21} = 0$, and all lower-order and forcing terms zero. The skeleton form is no restriction, locally and (when ellipticity provides a consistent global choice of nondegenerate first-order block) globally: any first-order $2 \times 2$ real elliptic system reduces to \eqref{eq:real-system} after a suitable choice of dependent-variable basis, by multiplying through by the inverse of the resulting matrix in front of $(v_x, v_y)$, whose determinant is non-zero by ellipticity. We work directly with \eqref{eq:real-system} throughout. This skeleton form is also the starting point of Vekua \cite[eq.~(7.5)]{vekua} and Bojarski \cite[eq.~(2.1)]{bojarski}; the divergence between their complex rewritings and ours occurs at the bundling step, not at the choice of skeleton.
\end{remark}

\section{Derivation of the Beltrami--Vekua form}\label{sec:BV-derivation}

We derive \eqref{eq:BV-intro} from \eqref{eq:real-system} through a seven-step pipeline. Each step is algebraic or consists of differentiating given data; no PDE is solved at any stage. We state each step as concisely as journal form permits; the construction passes through the moving algebra of variable elliptic structures developed in \cite{ves}, and we refer there for further context on the geometric meaning of the intermediate objects.

\subsection{Step 1: Structure data}\label{ssec:structure}

From the four principal-part coefficients we form the dimensionless ratios
\begin{equation}\label{eq:alpha-beta}
\alpha := \frac{a_{22}}{a_{11}}, \qquad \beta := -\frac{a_{12} + a_{21}}{a_{11}},
\end{equation}
and the \emph{structure polynomial} $p(s) := s^2 + \beta s + \alpha$. Its discriminant is $\Delta := 4\alpha - \beta^2$, and ellipticity \eqref{eq:ellipticity} is equivalent to $\Delta > 0$. The upper-half-plane root,
\begin{equation}\label{eq:lambda-def}
\lambda := \frac{-\beta + i\sqrt{\Delta}}{2} \in \HH^+,
\end{equation}
satisfies $\lambda^2 + \beta\lambda + \alpha = 0$ in $\C$, with $\lambda + \bar\lambda = -\beta$ and $\lambda\bar\lambda = \alpha$. The Cayley transform sends $\lambda$ to
\begin{equation}\label{eq:mu-def}
\mu := \frac{\lambda - i}{\lambda + i} \in \D,
\end{equation}
the \emph{Beltrami coefficient}. Two identities of $\mu$ in terms of $\lambda$ will be essential at Step~7:
\begin{equation}\label{eq:cayley}
\frac{1 + i\lambda}{1 - i\lambda} = -\mu, \qquad \frac{2}{1 - i\lambda} = 1 - \mu.
\end{equation}

\subsection{Step 2: The intrinsic obstruction}\label{ssec:obstruction}

The structure polynomial defines, at each $(x,y) \in \Omega$, the commutative $\R$-algebra
\begin{equation}\label{eq:Az-def}
\Az := \R[X]/(X^2 + \beta X + \alpha),
\end{equation}
with generator class $\ii := [X]$ satisfying $\ii^2 = -\beta\ii - \alpha$. Because $\alpha, \beta$ depend on $(x,y)$, so does $\ii$: the generator \emph{moves} with position. The natural measure of this movement is
\begin{equation}\label{eq:G-def}
G := \ii_x + \ii\, \ii_y \in \Az.
\end{equation}
Writing $G = G_1 + G_2 \ii$ with $G_1, G_2$ real, implicit differentiation of $\ii^2 + \beta\ii + \alpha = 0$ followed by inversion of the factor $2\ii + \beta$ in $\Az$ gives
\begin{equation}\label{eq:G-explicit}
G_1 = \frac{\beta P - 2\alpha Q}{\Delta}, \qquad G_2 = \frac{2P - \beta Q}{\Delta},
\end{equation}
where
\begin{equation}\label{eq:PQ-def}
P := \alpha_x - \alpha\beta_y, \qquad Q := \beta_x + \alpha_y - \beta\beta_y.
\end{equation}
The pair $(G_1, G_2)$ admits a parallel description in the spectral plane: a direct calculation shows that the transport defect $T := \lambda_x + \lambda\lambda_y$, computed in $\C$, satisfies $T = G_1 + G_2 \lambda$. We call $G$ the \emph{intrinsic obstruction}; the structure is \emph{rigid} when $G \equiv 0$. The smoothness hypothesis of Definition~\ref{def:smooth-class} is precisely what makes $G$ a well-defined $C^0$ section of $\Az$.

\subsection{Step 3: Canonical form}\label{ssec:canonical}

The change of dependent variables
\begin{equation}\label{eq:UV-def}
U := a_{22}\, u, \qquad V := v - a_{12}\, u
\end{equation}
brings \eqref{eq:real-system} to
\begin{equation}\label{eq:canonical}
\begin{cases}
U_x - \alpha\, V_y + a\, U + b\, V = f, \\[4pt]
V_x + U_y - \beta\, V_y + c\, U + d\, V = g,
\end{cases}
\end{equation}
with coefficients given by
\begin{equation}\label{eq:canonical-coeffs}
\begin{aligned}
a &= \tfrac{1}{a_{11}}\bigl[a_{13} + a_{12}a_{14} - (a_{12})_y\bigr] - \tfrac{(a_{22})_x}{a_{22}}, \\
b &= \alpha\, a_{14}, \\
c &= \tfrac{1}{a_{22}}\bigl[a_{23} + a_{12}a_{24} + (a_{12})_x - (a_{22})_y\bigr] - \tfrac{(a_{12}+a_{21})}{a_{11}a_{22}}\bigl[a_{13} + a_{12}a_{14} - (a_{12})_y\bigr], \\
d &= a_{24} + \beta\, a_{14}, \\
f &= \alpha\, f_1, \qquad g = f_2 + \beta\, f_1.
\end{aligned}
\end{equation}

\medskip
The substitution \eqref{eq:UV-def} together with the structure-polynomial algebra $\Az = \R[X]/(X^2 + \beta X + \alpha)$ was introduced in the author's earlier note \cite{alayon-solarz2011}. The construction proceeds through a specific sequence of moves, none of which is forced by general principles: first the substitution $U = a_{22}u$, $V = v - a_{12}u$, which is engineered to exploit the positivity $a_{22} > 0$ guaranteed by ellipticity; then a multiplication of the first equation by $\beta$ and addition to the second, which clears the $U_x$ term from the second line at the cost of introducing a $V_y$ term carrying the right coefficient; then a multiplication of the first equation by $\alpha$, which renders the principal part of the system precisely the symbol of the variable Cauchy--Riemann operator on $\Az = \R[X]/(X^2 + \beta X + \alpha)$. The choice $\alpha = a_{22}/a_{11}$ and $\beta = -(a_{12}+a_{21})/a_{11}$ is then forced --- not chosen --- by the requirement that the resulting principal part match the variable Cauchy--Riemann symbol; ellipticity becomes the discriminant condition $4\alpha - \beta^2 > 0$. The displayed lower-order coefficients in \eqref{eq:canonical-coeffs} extend the 2011 derivation, which treated the homogeneous case ($a_{i3} = a_{i4} = 0$, $f_k = 0$); the bookkeeping for the inhomogeneous and lower-order data follows the same row-operation pattern.

The 2011 note treats implicitly the case where the pair $\alpha$ and $\beta$ is effectively rigid, ($G \equiv 0$) --- equivalently, the case where the moving generator of the structure polynomial is holomorphic --- and interprets the resulting equation as a parameter-depending Vekua equation in the sense of Tutschke--Vanegas \cite{tutschke-vanegas}. In the rigid setting no obstruction arises and Steps~2 and~4 of the present pipeline are vacuous; the derivation passes directly from \eqref{eq:canonical} to a Vekua equation in $\Az$ via the bundling of Step~5. Thus the pipeline embeds that rigid case as a special instance and supplies the moving-generator structure (Steps~2 and~4) that makes the variable-coefficient case --- where $\alpha, \beta$ depend on position and the generator $\ii$ of $\Az$ correspondingly moves --- classically tractable. Every dependence on derivatives of the $a_{ij}$ produced by the row operations of \eqref{eq:UV-def}--\eqref{eq:canonical} is concentrated in the scalar coefficients $a$ and $c$; the additional first-derivative dependence introduced at Step~2 by the obstruction $G$ is what is new in the variable-coefficient case.
\subsection{Step 4: Modified coefficients}\label{ssec:modified}

We set
\begin{equation}\label{eq:modified}
\tilde b := b - G_1, \qquad \tilde d := d - G_2.
\end{equation}
The motivation will be visible at Step~5: when bundling $W = U + \ii V$, the product rule for $\dbar (\ii V)$ produces a term $\frac{1}{2} G\, V$, and the modification absorbs it. In the rigid case ($G \equiv 0$), $\tilde b = b$ and $\tilde d = d$.

\subsection{Step 5: The fiber-algebra Vekua equation}\label{ssec:fiber-vekua}

We bundle $W := U + \ii V \in \Az$ and apply the variable Cauchy--Riemann operator
\begin{equation}\label{eq:dbar-Az}
\dbar := \tfrac{1}{2}(\partial_x + \ii\, \partial_y).
\end{equation}
Using $\ii^2 = -\beta\ii - \alpha$ and $G = \ii_x + \ii\ii_y$, the product rule yields
\begin{equation}\label{eq:dbar-W}
\dbar W = \tfrac{1}{2}\bigl[(U_x - \alpha V_y) + \ii(U_y + V_x - \beta V_y) + G\, V\bigr].
\end{equation}
Substituting \eqref{eq:canonical} for the bracketed terms, decomposing $G\,V$ into its real and $\ii$-components, and using the modified coefficients \eqref{eq:modified} converts \eqref{eq:dbar-W} into the matching identity
\begin{equation}\label{eq:dbar-W-final}
\dbar W = -\tfrac{1}{2}\bigl[(aU + \tilde b\, V) + \ii(cU + \tilde d\, V)\bigr] + \tfrac{1}{2}(f + \ii g).
\end{equation}

\medskip
We seek to recast \eqref{eq:dbar-W-final} as a Vekua equation
\begin{equation}\label{eq:fiber-vekua-target}
\dbar W + A\, W + B\, \widehat W = F
\end{equation}
in $\Az$, where
\begin{equation}\label{eq:hatW-def}
\widehat W := U + \hat\ii\, V, \qquad \hat\ii := -\beta - \ii
\end{equation}
is the elliptic conjugate (the other root of the structure polynomial; in the standard case $\alpha = 1, \beta = 0$ this recovers $\bar W$). Write $A = A_0 + A_1\ii$ and $B = B_0 + B_1\ii$ with $A_j, B_j$ real-valued. Using $\ii^2 = -\beta\ii - \alpha$:
\begin{align*}
AW &= (A_0 + A_1\ii)(U + \ii V) = (A_0 U - \alpha A_1 V) + (A_0 V + A_1 U - \beta A_1 V)\,\ii, \\
B\widehat W &= (B_0 + B_1\ii)(U - (\beta + \ii)V) = (B_0 U - \beta B_0 V + \alpha B_1 V) + (B_1 U - B_0 V)\,\ii,
\end{align*}
where the $B\widehat W$ computation uses $-B_1 V \ii^2 = \alpha B_1 V + \beta B_1 V \ii$ to cancel a $-\beta B_1 V \ii$ term. Matching $AW + B\widehat W$ against the right-hand side of \eqref{eq:dbar-W-final} (with sign flipped, since \eqref{eq:fiber-vekua-target} has $+AW+B\widehat W$) yields four real equations:
\begin{equation}\label{eq:4x4-matching}
\begin{aligned}
\text{[$U$-real]:} &\quad A_0 + B_0 = a/2, \\
\text{[$V$-real]:} &\quad -\alpha A_1 - \beta B_0 + \alpha B_1 = \tilde b/2, \\
\text{[$U$-$\ii$]:} &\quad A_1 + B_1 = c/2, \\
\text{[$V$-$\ii$]:} &\quad A_0 - \beta A_1 - B_0 = \tilde d/2.
\end{aligned}
\end{equation}
The first and fourth equations determine $A_0$ and $B_0$ in terms of $A_1$:
\[
A_0 = \frac{a + \tilde d}{4} + \frac{\beta}{2}A_1, \qquad B_0 = \frac{a - \tilde d}{4} - \frac{\beta}{2}A_1.
\]
The third gives $B_1 = c/2 - A_1$. Substituting into the second and using $\Delta = 4\alpha - \beta^2$:
\[
-2\Delta A_1 = 2\tilde b + \beta(a - \tilde d) - 2\alpha c.
\]
With the asymmetry parameters
\begin{equation}\label{eq:sigma-tau}
\sigma := a - \tilde d, \qquad \tau := \alpha c - \tilde b,
\end{equation}
this reads $-2\Delta A_1 = -2\tau + \beta\sigma$, hence $A_1 = (2\tau - \beta\sigma)/(2\Delta)$. Back-substituting yields
\begin{equation}\label{eq:ABcoeffs}
\begin{aligned}
A_0 &= \frac{a + \tilde d}{4} + \frac{\beta(2\tau - \beta\sigma)}{4\Delta}, & A_1 &= \frac{2\tau - \beta\sigma}{2\Delta}, \\[4pt]
B_0 &= \frac{\sigma}{4} - \frac{\beta(2\tau - \beta\sigma)}{4\Delta}, & B_1 &= \frac{c\Delta - 2\tau + \beta\sigma}{2\Delta},
\end{aligned}
\end{equation}
and $F = (f + \ii g)/2$.

\subsection{Step 6: The spectral Vekua equation}\label{ssec:spectral-vekua}

The map $\ii \mapsto \lambda$ is an algebra homomorphism $\Az \to \C$ at each point, since $\lambda$ satisfies $\lambda^2 + \beta\lambda + \alpha = 0$ in $\C$. Under this map, $\hat\ii = -\beta - \ii \mapsto -\beta - \lambda = \bar\lambda$, the elliptic conjugate becomes the ordinary complex conjugate, and \eqref{eq:fiber-vekua-target} becomes
\begin{equation}\label{eq:spectral-vekua}
\dbar_\lambda\, w + A_\lambda\, w + B_\lambda\, \bar w = F_\lambda,
\end{equation}
where $w = U + \lambda V$, $\dbar_\lambda := \tfrac{1}{2}(\partial_x + \lambda\partial_y)$, and the coefficients in $\C$ are
\begin{align}
A_\lambda &= \frac{(2\alpha - \beta^2)\sigma + \Delta\tilde d + \beta\tau}{2\Delta} + \frac{2\tau - \beta\sigma}{2\Delta}\, \lambda, \label{eq:Alambda} \\
B_\lambda &= \frac{2\alpha\sigma - \beta\tau}{2\Delta} + \frac{\Delta c - 2\tau + \beta\sigma}{2\Delta}\, \lambda, \label{eq:Blambda} \\
F_\lambda &= \tfrac{1}{2}(f + \lambda g) = \tfrac{1}{2}\bigl[\alpha f_1 + \lambda(f_2 + \beta f_1)\bigr]. \label{eq:Flambda}
\end{align}
A direct addition shows $A_\lambda + B_\lambda = (a + c\lambda)/2$.

\subsection{Step 7: The Beltrami--Vekua form}\label{ssec:beltrami-vekua}

In conformal coordinates $\partial_x = \partial_z + \partial_{\bar z}$, $\partial_y = i(\partial_z - \partial_{\bar z})$, so
\[
w_x + \lambda w_y = (1 + i\lambda)\, w_z + (1 - i\lambda)\, w_{\bar z}.
\]
Factoring $(1 - i\lambda)$ and using the first Cayley identity in \eqref{eq:cayley}:
\begin{equation}\label{eq:intertwining-mid}
w_x + \lambda w_y = (1 - i\lambda)\bigl[w_{\bar z} - \mu\, w_z\bigr].
\end{equation}
Therefore $\dbar_\lambda = \tfrac{1}{2}(1 - i\lambda)(\partial_{\bar z} - \mu\partial_z)$, and the second Cayley identity gives $\tfrac{1}{2}(1 - i\lambda) = 1/(1-\mu)$. Substituting into \eqref{eq:spectral-vekua} and multiplying through by $(1-\mu)$:
\begin{equation}\label{eq:BV-final}
\boxed{\;w_{\bar z} - \mu\, w_z + \Acal\, w + \Bcal\, \bar w = \Fcal,\;}
\end{equation}
with
\begin{equation}\label{eq:BV-coeffs}
\Acal = (1-\mu)\, A_\lambda, \qquad \Bcal = (1-\mu)\, B_\lambda, \qquad \Fcal = (1-\mu)\, F_\lambda.
\end{equation}

\medskip
This completes the derivation. The data $(\mu, \Acal, \Bcal, \Fcal)$ are pointwise-defined continuous functions on $\Omega$, with regularity inherited from the hypotheses of Definition~\ref{def:smooth-class}: $\mu$ is $C^1$ (a function of the principal-part data alone), and $\Acal, \Bcal, \Fcal$ are $C^0$ (built from $C^0$ lower-order data, $C^0$ forcing, and the at-most-once-differentiated principal part). They are computed from the $a_{ij}, f_k$ and the first partial derivatives of the principal-part coefficients through the explicit formulas \eqref{eq:alpha-beta}--\eqref{eq:BV-coeffs}.

\medskip
\noindent
\textbf{Special cases.} When the principal part is constant (so $G = 0$), the formulas simplify substantially: $r := (a_{22})_x/a_{22} = 0$ and the modified coefficients equal the unmodified. When in addition $a_{i3} = a_{i4} = 0$, all of $a, b, c, d$ vanish, so $\Acal = \Bcal = \Fcal = 0$ and \eqref{eq:BV-final} reduces to the pure Beltrami equation $w_{\bar z} - \mu w_z = 0$. When $\alpha = 1, \beta = 0$, we have $\mu = 0$ and \eqref{eq:BV-final} reduces to the classical Vekua equation $w_{\bar z} + \Acal w + \Bcal \bar w = \Fcal$. 

\section{Gauge equivalence}\label{sec:gauge}

We work henceforth on the data side, taking \eqref{eq:BV-final} as our starting point. Let $\BV(\Omega)$ denote the space of BV data $D = (\mu, \Acal, \Bcal, \Fcal)$ with $\mu, \Acal, \Bcal, \Fcal$ continuous on $\Omega$ and $|\mu(z)| < 1$ pointwise.

\begin{definition}[Gauge]\label{def:gauge}
A \emph{gauge} on $\Omega$ is a function $\phi \in C^1(\Omega; \C^*)$, i.e., $C^1$, complex-valued, and nowhere zero. The set of all gauges is denoted $\Gcal(\Omega)$; it is a group under pointwise multiplication.
\end{definition}

\begin{proposition}[Gauge action]\label{prop:gauge-action}
Under $w \mapsto \tilde w := \phi w$ for $\phi \in \Gcal(\Omega)$, the equation \eqref{eq:BV-final} transforms to a new BV equation $\phi \cdot D = (\mu', \Acal', \Bcal', \Fcal')$ with
\begin{equation}\label{eq:gauge-action}
\mu' = \mu, \qquad \Acal' = \Acal - \frac{\phi_{\bar z}}{\phi} + \mu\,\frac{\phi_z}{\phi}, \qquad \Bcal' = \Bcal\,\frac{\phi}{\bar\phi}, \qquad \Fcal' = \phi\, \Fcal.
\end{equation}
The map $D \mapsto \phi \cdot D$ defines a left action of $\Gcal(\Omega)$ on $\BV(\Omega)$, and the orbit equivalence
\[
D_1 \sim_g D_2 \;\iff\; \exists \phi \in \Gcal(\Omega):\, \phi \cdot D_1 = D_2
\]
is an equivalence relation on $\BV(\Omega)$.
\end{proposition}

\begin{proof}
From $w = \tilde w/\phi$ we get $w_{\bar z} = \tilde w_{\bar z}/\phi - \tilde w \phi_{\bar z}/\phi^2$, $w_z = \tilde w_z/\phi - \tilde w \phi_z/\phi^2$, $\bar w = \overline{\tilde w}/\bar\phi$. Substituting into \eqref{eq:BV-final} and multiplying by $\phi$ gives \eqref{eq:gauge-action}. The action axioms follow from the additivity of logarithmic derivatives on products and the multiplicativity of $\phi/\bar\phi$.
\end{proof}

The most consequential feature of \eqref{eq:gauge-action} is the asymmetric way $\Acal$, $\Bcal$, $\Fcal$ transform.

\begin{itemize}
\item $\Acal$ shifts \emph{additively} by logarithmic derivatives of $\phi$.
\item $\Bcal$ is multiplied by $\phi/\bar\phi$, which is unimodular: writing $\phi = re^{i\theta}$, $\phi/\bar\phi = e^{2i\theta}$, hence
\begin{equation}\label{eq:Bmod-gauge-inv}
|\Bcal'| = |\Bcal|.
\end{equation}
\item $\Fcal$ is multiplied by $\phi$, which has arbitrary modulus $r > 0$, so $|\Fcal|$ is \emph{not} gauge-invariant.
\end{itemize}

The gauge-invariance of $|\Bcal|$ as a function on $\Omega$ is the algebraic fulcrum on which everything in Section~\ref{sec:invariant} turns. The contrast with $|\Fcal|$ will reappear when we ask whether $\Fcal$ supports an analogous mass-type density.

\section{Diffeomorphism equivalence}\label{sec:diffeo}

The second symmetry is change of coordinates on the domain. Unlike the gauge action, which fixed $\Omega$ and altered only the unknown, the diffeomorphism action moves between potentially distinct domains.

\begin{definition}[Diffeomorphism, Jacobian]\label{def:diffeo}
A \emph{(orientation-preserving) diffeomorphism} between domains $\Omega_1, \Omega_2 \subset \C$ is a $C^1$ bijection $\Phi: \Omega_1 \to \Omega_2$ with $C^1$ inverse and Jacobian
\begin{equation}\label{eq:jacobian}
J := |\Phi_z|^2 - |\Phi_{\bar z}|^2 > 0
\end{equation}
everywhere on $\Omega_1$.
\end{definition}

\begin{proposition}[Diffeomorphism pullback]\label{prop:diffeo-action}
Let $\Phi: \Omega_1 \to \Omega_2$ be a diffeomorphism and $D = (\mu, \Acal, \Bcal, \Fcal) \in \BV(\Omega_2)$. Set
\begin{equation}\label{eq:K-def}
K := \Phi_z + (\mu \circ \Phi)\,\overline{\Phi_{\bar z}}.
\end{equation}
Then $K$ is nowhere zero on $\Omega_1$, and $w \circ \Phi$ satisfies the BV equation $\Phi^* D = (\mu^*, \Acal^*, \Bcal^*, \Fcal^*)$ on $\Omega_1$ with
\begin{equation}\label{eq:diffeo-pullback}
\mu^* = \frac{\Phi_{\bar z} + (\mu\circ\Phi)\,\overline{\Phi_z}}{K}, \qquad
\Acal^* = \frac{J\,(\Acal \circ \Phi)}{K}, \qquad
\Bcal^* = \frac{J\,(\Bcal\circ\Phi)}{K}, \qquad
\Fcal^* = \frac{J\,(\Fcal\circ\Phi)}{K}.
\end{equation}
The pulled-back $\mu^*$ satisfies $|\mu^*| < 1$, so $\Phi^* D \in \BV(\Omega_1)$.
\end{proposition}

\begin{proof}
Write $\tilde w := w \circ \Phi$. The chain rule
\[
\begin{pmatrix} \tilde w_z \\ \tilde w_{\bar z} \end{pmatrix} = \begin{pmatrix} \Phi_z & \overline{\Phi_{\bar z}} \\ \Phi_{\bar z} & \overline{\Phi_z} \end{pmatrix} \begin{pmatrix} w_\zeta \circ \Phi \\ w_{\bar\zeta} \circ \Phi \end{pmatrix}
\]
inverts (with determinant $J > 0$) to
\[
w_\zeta \circ \Phi = \frac{\overline{\Phi_z}\,\tilde w_z - \overline{\Phi_{\bar z}}\,\tilde w_{\bar z}}{J}, \qquad w_{\bar\zeta} \circ \Phi = \frac{-\Phi_{\bar z}\,\tilde w_z + \Phi_z\,\tilde w_{\bar z}}{J}.
\]
Substituting into $w_{\bar\zeta} - \mu w_\zeta + \Acal w + \Bcal\bar w = \Fcal$ composed with $\Phi$, multiplying by $J$, and dividing by $K$ (the resulting coefficient of $\tilde w_{\bar z}$) gives \eqref{eq:diffeo-pullback}.

The closure identity
\begin{equation}\label{eq:closure}
|K|^2 - |\Phi_{\bar z} + (\mu\circ\Phi)\,\overline{\Phi_z}|^2 = (1 - |\mu\circ\Phi|^2)\, J
\end{equation}
follows by expanding both squared moduli; the cross terms $2\Re((\overline{\mu\circ\Phi})\Phi_z\Phi_{\bar z})$ are identical and cancel, leaving $(|\Phi_z|^2 - |\Phi_{\bar z}|^2)(1 - |\mu\circ\Phi|^2) = (1-|\mu\circ\Phi|^2)J$. Both factors on the right being strictly positive, $|K|^2 \geq (1-|\mu\circ\Phi|^2)J > 0$, so $K$ is nowhere zero. Dividing both sides of \eqref{eq:closure} by $|K|^2$:
\begin{equation}\label{eq:conformal-identity}
1 - |\mu^*|^2 = (1 - |\mu\circ\Phi|^2) \,\frac{J}{|K|^2} > 0,
\end{equation}
hence $|\mu^*| < 1$.
\end{proof}

\begin{remark}
The identity \eqref{eq:conformal-identity} --- the \emph{conformal identity} --- is central to the proof of Theorem~\ref{thm:main-invariance} below: the multiplicative weight $J/|K|^2$ on the right is exactly what cancels in the diffeomorphism transformation of $|\Bcal|^2$, producing the invariant 2-form. The transformation law for $\mu$ in \eqref{eq:diffeo-pullback} is the classical Beltrami composition formula \cite{ahlfors}; specializing to $\mu \equiv 0$ recovers $\mu^* = \Phi_{\bar z}/\Phi_z$, the dilatation of $\Phi$. In particular $|\mu|$ is \emph{not} preserved by diffeomorphism pullback --- the entire content of Beltrami uniformization is that $|\mu|$ can be reduced to zero by an appropriate choice of $\Phi$.
\end{remark}

\begin{theorem}\label{thm:diffeo-equiv}
The relation
\[
D_1 \sim_d D_2 \;\iff\; \exists \text{ diffeomorphism } \Phi: \Omega_1 \to \Omega_2 \text{ such that } \Phi^* D_2 = D_1
\]
is an equivalence relation on the disjoint union of all $\BV(\Omega)$ over $\Omega \subset \C$.
\end{theorem}

The proof verifies reflexivity (identity diffeomorphism), symmetry (inverse), and transitivity (composition); the algebraic content is the consistency of \eqref{eq:diffeo-pullback} under composition, which follows from the chain rule.

\medskip
\noindent
\textbf{The combined equivalence.} The relations $\sim_g$ and $\sim_d$ generate a coarser equivalence relation
\begin{equation}\label{eq:combined-equiv}
D_1 \sim D_2 \;\iff\; \text{a finite chain of $\sim_g$ and $\sim_d$ steps connects } D_1 \text{ to } D_2.
\end{equation}
Equivalently, $\sim$ is the orbit equivalence of the joint action of the gauge group and the diffeomorphism groupoid. The remainder of the paper concerns this combined relation.

\section{The classical reduction inside Beltrami--Vekua}\label{sec:reduction}

Before introducing the main invariant, we record a structural application of the equivalence theory: the classical two-stage reduction of Vekua \cite{vekua} --- (i) solve the Beltrami equation to uniformize, (ii) gauge away the resulting $w$-coupling --- becomes a single composite operation in the BV framework, with the gauge step reducing after one Beltrami solve to a flat $\bar\partial$-equation handled by the Cauchy transform.

\subsection{Reduction to Vekua via uniformization}

\begin{proposition}[Reduction to Vekua]\label{prop:reduction-to-vekua}
Let $D = (\mu, \Acal, \Bcal, \Fcal) \in \BV(\Omega)$ with $\|\mu\|_\infty < 1$, and suppose there exists a $C^1$ diffeomorphism $\Phi: \Omega \to \Omega'$ satisfying the Beltrami equation
\begin{equation}\label{eq:Beltrami-uniformizing}
\Phi_{\bar z} = \mu\, \Phi_z.
\end{equation}
Set $\Psi := \Phi^{-1}: \Omega' \to \Omega$. Then $\Psi^* D$ is a Vekua equation on $\Omega'$:
\[
w_{\bar\zeta} + \Acal'\, w + \Bcal'\, \bar w = \Fcal',
\]
with $\mu' \equiv 0$ and
\begin{equation}\label{eq:Vekua-coeffs}
\Acal' = \overline{\Psi_\zeta}\,(\Acal \circ \Psi), \qquad \Bcal' = \overline{\Psi_\zeta}\,(\Bcal \circ \Psi), \qquad \Fcal' = \overline{\Psi_\zeta}\,(\Fcal \circ \Psi).
\end{equation}
In particular, $D \sim_d \Psi^* D$ is gauge-diffeomorphism equivalent to a classical Vekua equation.
\end{proposition}

\begin{proof}
Apply Proposition~\ref{prop:diffeo-action} with $\Psi$ in place of $\Phi$, taking $\Omega_1 = \Omega'$ and $\Omega_2 = \Omega$. Differentiating $\Phi(\Psi(\zeta)) = \zeta$ with respect to $\bar\zeta$ and using \eqref{eq:Beltrami-uniformizing} gives the conjugate identity
\begin{equation}\label{eq:conjugate-Beltrami}
\Psi_{\bar\zeta} = -(\mu\circ\Psi)\,\overline{\Psi_\zeta}.
\end{equation}
The pullback formula \eqref{eq:diffeo-pullback} gives $\mu' = (\Psi_{\bar\zeta} + (\mu\circ\Psi)\overline{\Psi_\zeta})/K$, whose numerator vanishes by \eqref{eq:conjugate-Beltrami}. Computing $J$ and $K$ from \eqref{eq:conjugate-Beltrami}: $|\Psi_{\bar\zeta}|^2 = |\mu\circ\Psi|^2|\Psi_\zeta|^2$, so $J = |\Psi_\zeta|^2(1 - |\mu\circ\Psi|^2)$ and $K = \Psi_\zeta(1 - |\mu\circ\Psi|^2)$. The ratio collapses: $J/K = \overline{\Psi_\zeta}$, giving \eqref{eq:Vekua-coeffs}.
\end{proof}

\begin{remark}[Existence of the uniformizing diffeomorphism]\label{rem:Beltrami-existence}
The hypothesis that a $C^1$ diffeomorphism solving \eqref{eq:Beltrami-uniformizing} exists is a classical regularity question. The measurable Riemann mapping theorem of Ahlfors and Bers \cite{ahlfors} produces, for any measurable $\mu$ with $\|\mu\|_\infty < 1$, a quasiconformal homeomorphism $\Phi$ satisfying \eqref{eq:Beltrami-uniformizing} in the distributional sense. Standard Schauder theory for the Beltrami equation then upgrades $\Phi$ to a classical solution under appropriate H\"older hypotheses on the coefficient. For the propositions of this section we adopt the standing assumption $\mu \in C^{1,\alpha}(\Omega)$ for some $\alpha \in (0,1)$ together with the uniform ellipticity bound $\|\mu\|_\infty < 1$, which together are more than sufficient to guarantee that $\Phi \in C^{2,\alpha}$ globally and that all subsequent compositions and coefficient transformations are well-defined classically.
\end{remark}

\subsection{The single-PDE-solve normal form}

We now combine Proposition~\ref{prop:reduction-to-vekua} with the classical fact that the lower-order coefficient $\Acal$ in a Vekua equation can be eliminated by a gauge transformation. The key observation is that the same uniformizing diffeomorphism $\Phi$ that handles the principal part also reduces the gauge step to a flat $\bar\partial$-equation on the uniformized domain.

\begin{proposition}[Gauging away $\Acal$ in a Vekua equation]\label{prop:gauge-away-A}
Let $\Omega \subset \C$ be bounded, and let $D = (0, \Acal, \Bcal, \Fcal) \in \BV(\Omega)$ be a Vekua equation with $\Acal \in C^{0,\alpha}(\overline\Omega)$ for some $\alpha \in (0,1)$. There exists a gauge $\phi \in \Gcal(\Omega)$ such that
\[
\phi \cdot D = (0,\, 0,\, \Bcal\,\tfrac{\phi}{\bar\phi},\, \phi \Fcal).
\]
Specifically, $\phi := \exp(T\Acal)$ where $T$ is the Cauchy transform $(T\Acal)(z) := -\tfrac{1}{\pi}\int_\Omega \tfrac{\Acal(\zeta)}{\zeta - z}\,d\xi\,d\eta$. (The boundedness of $\Omega$ guarantees convergence of the Cauchy integral; on unbounded domains the construction yields only local gauges on bounded subdomains.)
\end{proposition}

\begin{proof}
With $\mu = 0$, the gauge formula \eqref{eq:gauge-action} for $\Acal$ collapses to $\Acal' = \Acal - \phi_{\bar z}/\phi$. Setting $\Acal' = 0$ requires $\bar\partial \log\phi = \Acal$. For Hölder continuous $\Acal$ on a bounded $\Omega$, the Cauchy transform $T\Acal$ belongs to $C^1_{\mathrm{loc}}(\Omega)$ and satisfies $\bar\partial(T\Acal) = \Acal$ in $\Omega$ (see \cite[Ch.~I, \S 6]{vekua}; closure regularity up to $\partial\Omega$ would require additional boundary regularity which we do not need here). Hence $\phi := \exp(T\Acal)$ is a $C^1$ nowhere-vanishing function on $\Omega$ satisfying the equation.
\end{proof}

\begin{proposition}[Recycling the Beltrami uniformization]\label{prop:single-solve}
Let $\Omega \subset \C$ be a bounded domain, $D = (\mu, \Acal, \Bcal, \Fcal) \in \BV(\Omega)$ with $\mu \in C^{1,\alpha}$ for some $\alpha \in (0,1)$ and $\|\mu\|_\infty < 1$, and let $\Phi: \Omega \to \Omega'$ be a $C^{2,\alpha}$ diffeomorphism solving the Beltrami equation \eqref{eq:Beltrami-uniformizing}. Assume that $\Omega' = \Phi(\Omega)$ is bounded\footnote{This is automatic when $\Omega$ is bounded and $\Phi$ extends continuously to $\overline\Omega$, which holds under the stated regularity. For unbounded domains the construction below gives only \emph{local} gauges on bounded subdomains; a global statement requires an additional global $\bar\partial$-solvability or patching argument, which we do not pursue here.} and that the pulled-back coefficient $\Acal' := \overline{\Psi_\zeta}(\Acal\circ\Psi)$ is H\"older continuous on $\overline{\Omega'}$. Then the composite reduction of $D$ to the form $(0, 0, \tilde\Bcal, \tilde\Fcal)$ requires the solution of only one variable-principal-part PDE --- the Beltrami equation \eqref{eq:Beltrami-uniformizing} --- followed by the standard Cauchy-transform solution of a flat $\dbar$-equation on $\Omega'$. The same diffeomorphism $\Phi$ that uniformizes the principal part supplies the integrating factor for the gauge step:
\begin{equation}\label{eq:phi-recycled}
\phi := \exp\bigl((T\Acal') \circ \Phi\bigr).
\end{equation}
\end{proposition}

\begin{proof}
Starting from $D$ on $\Omega$, applying the gauge step \emph{before} the diffeomorphism step would require solving
\begin{equation}\label{eq:twisted-dbar}
(\partial_{\bar z} - \mu\,\partial_z)\,\log\phi = \Acal
\end{equation}
on $\Omega$, a $\bar\partial$-equation \emph{twisted} by the variable principal part $\mu$. A naive count predicts a second variable-coefficient PDE problem. We show that the substitution $\log\phi = g \circ \Phi$, with $\Phi$ the same uniformizing diffeomorphism that solves \eqref{eq:Beltrami-uniformizing}, reduces \eqref{eq:twisted-dbar} to the flat (constant-coefficient) $\bar\partial$-equation $g_{\bar\zeta} = \Acal'$ on $\Omega'$, solvable by the standard Cauchy transform.

Substitute the ansatz $\log\phi = g \circ \Phi$ for a function $g$ on $\Omega'$. The chain rule gives
\[
\partial_{\bar z}(g\circ\Phi) = (g_\zeta \circ \Phi)\,\Phi_{\bar z} + (g_{\bar\zeta} \circ \Phi)\,\overline{\Phi_z},
\qquad
\partial_z(g\circ\Phi) = (g_\zeta\circ\Phi)\,\Phi_z + (g_{\bar\zeta}\circ\Phi)\,\overline{\Phi_{\bar z}}.
\]
Therefore, using $\Phi_{\bar z} = \mu\Phi_z$ and $\overline{\Phi_{\bar z}} = \bar\mu\,\overline{\Phi_z}$,
\[
(\partial_{\bar z} - \mu\partial_z)(g\circ\Phi) = (g_\zeta\circ\Phi)(\Phi_{\bar z} - \mu\Phi_z) + (g_{\bar\zeta}\circ\Phi)(\overline{\Phi_z} - \mu\,\bar\mu\,\overline{\Phi_z}) = (1 - |\mu|^2)\,\overline{\Phi_z}\,(g_{\bar\zeta}\circ\Phi),
\]
the first parenthesis vanishing by \eqref{eq:Beltrami-uniformizing}. So \eqref{eq:twisted-dbar} reads
\[
(1 - |\mu|^2)\,\overline{\Phi_z}\,(g_{\bar\zeta}\circ\Phi) = \Acal.
\]
Composing with $\Psi$ (which inverts $\Phi$) and using the identity $\overline{\Phi_z(\Psi)} = 1/[\overline{\Psi_\zeta}(1 - |\mu\circ\Psi|^2)]$ --- a consequence of differentiating $\Phi(\Psi(\zeta)) = \zeta$ as in the proof of Proposition~\ref{prop:reduction-to-vekua} --- the equation pulls back through $\Phi$ to
\[
g_{\bar\zeta} = \overline{\Psi_\zeta}\,(\Acal\circ\Psi) = \Acal'
\]
on $\Omega'$. Since $\Omega'$ is bounded and $\Acal'$ is H\"older continuous on $\overline{\Omega'}$, the Cauchy transform on $\Omega'$ supplies $g := T\Acal' \in C^1_{\mathrm{loc}}(\Omega')$ with $g_{\bar\zeta} = \Acal'$ in $\Omega'$, and the gauge on $\Omega$ is \eqref{eq:phi-recycled}. The only variable-coefficient PDE solved is \eqref{eq:Beltrami-uniformizing}; the integrating factor on $\Omega$ is reconstructed from $\Phi$ together with the standard $\bar\partial$-inverse on the uniformized domain.
\end{proof}

\begin{corollary}[Normal form]\label{cor:normal-form}
Under the hypotheses of Proposition~\ref{prop:single-solve}, $D$ is gauge-diffeomorphism equivalent to
\[
(0, 0, \tilde\Bcal, \tilde\Fcal)
\]
on $\Omega'$, a Vekua equation with no $w$-coupling.
\end{corollary}

\subsection{Residual symmetries}

The normal form $(0, 0, \tilde\Bcal, \tilde\Fcal)$ is unique up to gauges and diffeomorphisms preserving both $\mu = 0$ and $\Acal = 0$. A gauge $\phi$ preserves $\Acal = 0$ iff $\phi_{\bar z}/\phi = 0$, i.e., $\phi$ is holomorphic. A diffeomorphism $\Phi$ preserves $\mu = 0$ iff $\Phi_{\bar z} = 0$, which combined with orientation-preservation makes $\Phi$ biholomorphic. The transformations preserving the normal-form conditions are therefore precisely holomorphic gauges and biholomorphic coordinate changes; their induced action on $(\tilde\Bcal, \tilde\Fcal)$, which is non-trivial, is not classified here.

\medskip
Corollary~\ref{cor:normal-form} situates classical pseudo-analytic function theory inside the BV framework: in the regular regime ($\mu \in C^{1,\alpha}$ with $\|\mu\|_\infty < 1$, bounded $\Omega$, H\"older pulled-back $\Acal'$, and existence of a global $C^{2,\alpha}$ Beltrami uniformizing diffeomorphism), every BV equivalence class admits a representative of the form $(0, 0, \tilde\Bcal, \tilde\Fcal)$, with residual freedom given by holomorphic gauges and biholomorphic coordinate changes. A full quotient-level classification --- in particular, the bookkeeping of uniformizing maps, residual symmetries, regularity strata, and domain topology --- is not pursued here; the BV form simply makes the underlying single-Beltrami-solve structure explicit.

\section{The analytic / pseudo-analytic dichotomy}\label{sec:dichotomy}

The transformation laws \eqref{eq:gauge-action} and \eqref{eq:diffeo-pullback} have an immediate consequence for the locus $\{\Bcal = 0\}$.

\begin{proposition}\label{prop:B-zero-invariant}
The locus $\{\Bcal(z) = 0\} \subset \Omega$ is preserved (as a subset, after the appropriate identification) under both gauge transformations and diffeomorphism pullbacks.
\end{proposition}

\begin{proof}
Under gauge, $\Bcal' = \Bcal\,\phi/\bar\phi$, and $\phi/\bar\phi$ is nowhere zero, so $\Bcal'(z) = 0$ iff $\Bcal(z) = 0$. Under diffeomorphism pullback, $\Bcal^* = (J/K)(\Bcal\circ\Phi)$, and $J/K$ is nowhere zero, so $\Bcal^*(\zeta) = 0$ iff $\Bcal(\Phi(\zeta)) = 0$.
\end{proof}

\begin{definition}[Analytic / pseudo-analytic]\label{def:analytic-class}
A BV equation $D \in \BV(\Omega)$ is \emph{analytic} if $\Bcal \equiv 0$ on $\Omega$, and \emph{pseudo-analytic} otherwise.
\end{definition}

By Proposition~\ref{prop:B-zero-invariant}, both classes are saturated under $\sim$, so analyticity is a genuine invariant property of equivalence classes. The terminology follows Vekua \cite{vekua} and Bers \cite{bers}: when $\Bcal \equiv 0$ the equation $w_{\bar z} - \mu w_z + \Acal w = \Fcal$ reduces (after uniformization and gauge, via Corollary~\ref{cor:normal-form}) to a flat $\bar\partial$-equation $h_{\bar\zeta} = \tilde\Fcal$ in the new coordinate. In the homogeneous case $\Fcal \equiv 0$, this means $h$ is holomorphic and the original $w$ has the form $\phi \cdot h$ with $\phi$ the recovered integrating factor --- structurally close to honestly analytic functions. With forcing, the solutions are not holomorphic but solve a flat inhomogeneous $\bar\partial$-equation, classically tractable by Cauchy transform. The case $\Bcal \not\equiv 0$ is the genuinely pseudo-analytic regime, where no such reduction to a flat $\bar\partial$-equation is possible within the allowed gauge--diffeomorphism equivalence.

The dichotomy is, however, only Boolean. It separates two saturated subsets of the gauge-diffeomorphism quotient but says nothing about the structure within either.

\section{The invariant 2-form}\label{sec:invariant}

We now refine the Boolean dichotomy of Section~\ref{sec:dichotomy} into a continuous invariant.

\subsection{Definition and main theorem}

\begin{definition}[The 2-form $\Theta$]\label{def:Theta}
For $D = (\mu, \Acal, \Bcal, \Fcal) \in \BV(\Omega)$ define
\begin{equation}\label{eq:Theta-def}
\Theta_D := \frac{|\Bcal|^2}{1 - |\mu|^2}\, dx\, dy.
\end{equation}
This is a non-negative 2-form on $\Omega$, equivalently a non-negative Borel measure.
\end{definition}

\begin{theorem}[Main invariance]\label{thm:main-invariance}
The 2-form $\Theta$ is intrinsic to the gauge-diffeomorphism equivalence class:
\begin{enumerate}
\item[(i)] (Gauge invariance.) For any $\phi \in \Gcal(\Omega)$ and any $D \in \BV(\Omega)$,
\[
\Theta_{\phi \cdot D} = \Theta_D \quad \text{on } \Omega.
\]
\item[(ii)] (Diffeomorphism covariance.) For any diffeomorphism $\Phi: \Omega_1 \to \Omega_2$ and any $D \in \BV(\Omega_2)$,
\[
\Theta_{\Phi^* D} = \Phi^*\, \Theta_D \quad \text{on } \Omega_1.
\]
\end{enumerate}
\end{theorem}

\begin{proof}
\emph{(i)} Under gauge, $|\mu'| = |\mu|$ and $|\Bcal'| = |\Bcal|$ (the unimodular factor $\phi/\bar\phi$ has modulus 1). Hence
\[
\Theta_{\phi\cdot D} = \frac{|\Bcal'|^2}{1 - |\mu'|^2}\, dx\, dy = \frac{|\Bcal|^2}{1 - |\mu|^2}\, dx\, dy = \Theta_D.
\]

\emph{(ii)} We compute $\Theta_{\Phi^* D}$ on $\Omega_1$ and compare with $\Phi^* \Theta_D$.

From \eqref{eq:diffeo-pullback}, $\Bcal^* = (J/K)(\Bcal\circ\Phi)$, so
\begin{equation}\label{eq:Bstar-mod}
|\Bcal^*|^2 = \frac{J^2}{|K|^2}\,|\Bcal\circ\Phi|^2.
\end{equation}
The conformal identity \eqref{eq:conformal-identity} gives $1 - |\mu^*|^2 = (1 - |\mu\circ\Phi|^2)\,J/|K|^2$. Dividing \eqref{eq:Bstar-mod} by this:
\begin{equation}\label{eq:Theta-pulled}
\frac{|\Bcal^*|^2}{1 - |\mu^*|^2} = \frac{(J^2/|K|^2)\,|\Bcal\circ\Phi|^2}{(1-|\mu\circ\Phi|^2)\,J/|K|^2} = \frac{J\,|\Bcal\circ\Phi|^2}{1 - |\mu\circ\Phi|^2},
\end{equation}
hence
\[
\Theta_{\Phi^* D} = \frac{J\,|\Bcal\circ\Phi|^2}{1 - |\mu\circ\Phi|^2}\,dx_1\,dy_1 \quad \text{on } \Omega_1.
\]
On the other hand, the pullback of $\Theta_D = |\Bcal|^2/(1-|\mu|^2)\,dx_2\,dy_2$ by $\Phi$ multiplies the area form by $J$, giving
\[
\Phi^*\Theta_D = \frac{|\Bcal\circ\Phi|^2}{1 - |\mu\circ\Phi|^2}\,J\,dx_1\,dy_1,
\]
which agrees with $\Theta_{\Phi^*D}$.
\end{proof}

\begin{proposition}[Uniqueness of the density]\label{prop:density-uniqueness}
Let $H: [0,1) \to (0, \infty)$, and consider the density
\[
\Theta^{(H)}_D := |\Bcal|^2\, H(|\mu|^2)\, dx\, dy.
\]
Then $\Theta^{(H)}$ is gauge-invariant and diffeomorphism-covariant in the
sense of Theorem~\ref{thm:main-invariance} if and only if
\[
H(s) = \frac{C}{1 - s}, \qquad C > 0.
\]
In particular, up to a positive multiplicative constant, $\Theta$ is the
unique density of the displayed multiplicative form
$|\Bcal|^2\,H(|\mu|^2)\,dx\,dy$ that is intrinsic to the gauge-diffeomorphism
equivalence class. (No regularity hypothesis on $H$ is needed; the
functional equation derived below forces the stated form pointwise.)
\end{proposition}

\begin{proof}
\emph{Sufficiency.} Theorem~\ref{thm:main-invariance} is exactly the case
$C = 1$, and any positive constant multiple of an invariant density is
itself invariant.

\emph{Necessity.} Gauge-invariance of $\Theta^{(H)}$ follows from the gauge
laws \eqref{eq:gauge-action} for any choice of $H$, since $|\Bcal|$ and
$|\mu|$ are both gauge-invariant. The constraint comes from
diffeomorphism covariance.

Fix any $D = (\mu, \Acal, \Bcal, \Fcal) \in \BV(\Omega_2)$ with $\Bcal$
nowhere zero. Let $\Phi: \Omega_1 \to \Omega_2$ be an
orientation-preserving diffeomorphism. By
Proposition~\ref{prop:diffeo-action} and identity~\eqref{eq:Bstar-mod},
\[
|\Bcal^*|^2 = \frac{J^2}{|K|^2}\, |\Bcal \circ \Phi|^2,
\]
and by the conformal identity~\eqref{eq:conformal-identity},
\[
|\mu^*|^2 = 1 - (1 - |\mu\circ\Phi|^2)\,\frac{J}{|K|^2}.
\]
Therefore
\[
\Theta^{(H)}_{\Phi^*D} = \frac{J^2}{|K|^2}\, |\Bcal\circ\Phi|^2\,
H\!\left(1 - (1-|\mu\circ\Phi|^2)\,\frac{J}{|K|^2}\right) dx_1\,dy_1.
\]
The pullback of $\Theta^{(H)}_D = |\Bcal|^2\, H(|\mu|^2)\, dx_2\,dy_2$
multiplies the area form by $J$, giving
\[
\Phi^* \Theta^{(H)}_D = J\, |\Bcal\circ\Phi|^2\, H(|\mu\circ\Phi|^2)\,
dx_1\,dy_1.
\]
Equating the coefficients of $dx_1\,dy_1$ and dividing by the nowhere-zero
factor $|\Bcal\circ\Phi|^2$:
\begin{equation}\label{eq:H-functional-clean}
\frac{J}{|K|^2}\, H\!\left(1 - (1-|\mu\circ\Phi|^2)\,\frac{J}{|K|^2}\right)
= H(|\mu\circ\Phi|^2).
\end{equation}
Set $s := |\mu \circ \Phi|^2 \in [0, 1)$ and $r := J/|K|^2$. The
closure identity~\eqref{eq:closure} gives
$|K|^2 \geq (1 - |\mu\circ\Phi|^2)\,J$, hence $r \leq 1/(1-s)$, with
equality precisely when $\Phi_{\bar z} + (\mu\circ\Phi)\overline{\Phi_z}
= 0$ at the point in question --- equivalently, when $\mu^* = 0$ there.
This conformal endpoint $r = 1/(1-s)$ is realizable: for any
$s_0 \in [0,1)$, take $\mu$ constant of modulus $\sqrt{s_0}$, $\Bcal
\equiv 1$, and choose $\Phi$ at a chosen point so that
$\Phi_{\bar z} = -(\mu\circ\Phi)\,\overline{\Phi_z}$ holds there
(this requires only $|\Phi_{\bar z}/\Phi_z| = \sqrt{s_0}$ at the point,
which is compatible with $J > 0$). Equation~\eqref{eq:H-functional-clean}
therefore holds at every $s \in [0,1)$ for at least the value
$r = 1/(1-s)$:
\begin{equation}\label{eq:H-functional-rs}
r\, H(1 - (1-s)r) = H(s).
\end{equation}
At $r = 1/(1-s)$, the argument of $H$ on the left becomes
$1 - (1-s)\cdot 1/(1-s) = 0$, so
\[
\frac{1}{1-s}\, H(0) = H(s),
\]
giving $H(s) = H(0)/(1-s)$ for every $s \in [0, 1)$. Setting $C := H(0) > 0$
yields the claimed form.

One verifies directly that this $H$ satisfies
\eqref{eq:H-functional-rs} for every admissible pair $(s, r)$, not only
at the conformal endpoint: $r \cdot C/((1-s)r) = C/(1-s) = H(s)$. Hence
the equation is consistent throughout the full range of $r$, and no
sub-conformal data are needed in the argument.
\end{proof}

\subsection{Why \texorpdfstring{$\Fcal$}{F} does not yield an analogous mass density}

\begin{lemma}\label{lem:no-F-invariant}
There is no expression of the form
\[
\Theta_F := \frac{|\Fcal|^2}{H(\mu, \Acal, \Bcal)}\,dx\,dy,
\]
with $H$ a continuous, nowhere-vanishing function of its arguments on the relevant data range, that is gauge-invariant for all non-trivial $\Fcal$.
\end{lemma}

\begin{proof}
Take a positive real constant gauge $\phi \equiv r$ for some $r > 0$. Since $\phi_z = \phi_{\bar z} = 0$ and $r/\bar r = 1$, the gauge transformation law \eqref{eq:gauge-action} reduces to
\[
\mu' = \mu, \qquad \Acal' = \Acal, \qquad \Bcal' = \Bcal, \qquad \Fcal' = r\,\Fcal.
\]
Thus $\mu, \Acal, \Bcal$ are \emph{exactly} unchanged (not merely up to phase), so any continuous $H(\mu, \Acal, \Bcal)$ is unchanged, while $|\Fcal|^2 \mapsto r^2\,|\Fcal|^2$. Any expression $\Theta_F = |\Fcal|^2/H(\mu, \Acal, \Bcal)\,dx\,dy$ would therefore transform as $\Theta_F \mapsto r^2\,\Theta_F$ for every $r > 0$, hence cannot be gauge-invariant unless $\Fcal \equiv 0$.
\end{proof}

\begin{remark}[Scope of the lemma]\label{rem:F-scope}
Lemma~\ref{lem:no-F-invariant} rules out only an $\Fcal$-invariant of the displayed algebraic form: a non-negative density quadratic in $\Fcal$ and divided by a continuous function of the remaining BV data. It does not rule out invariants of $\Fcal$ at the level of zero-locus or differential structure. In particular, the zero locus of $\Fcal$ is preserved by both equivalences --- under gauge $\Fcal' = \phi\Fcal$ with $\phi$ nowhere zero, and under diffeomorphism pullback $\Fcal^* = (J/K)(\Fcal\circ\Phi)$ with $J/K$ nowhere zero --- so $\{\Fcal = 0\}$ is a Boolean invariant analogous to $\{\Bcal = 0\}$. The structural content of the lemma is the absence of a \emph{continuous mass-type} invariant of $\Fcal$ at the algebraic level; finer differential or cohomological invariants of the forcing remain a possibility we do not pursue here.
\end{remark}

\begin{remark}[Source of the asymmetry]\label{rem:asymmetry-source}
The gauge action multiplies $\Bcal$ by the unimodular $\phi/\bar\phi$ and $\Fcal$ by the unconstrained $\phi$. The unimodular factor cancels in $|\Bcal|^2$; no companion in the BV data cancels the $|\phi|$ factor in $|\Fcal|$. Structurally, the difference traces back to how each coefficient appears in \eqref{eq:BV-final}: $\Bcal$ multiplies $\bar w$, which under $w \mapsto \phi w$ becomes $\bar\phi \cdot \overline{\tilde w}$, contributing $\bar\phi$ to be balanced against the overall $\phi$ multiplication of the equation; $\Fcal$ has no $w$-derivative or $w$-factor to absorb a factor of $\phi$, so it transforms by $\phi$ alone.
\end{remark}

\section{The pseudo-analytic mass}\label{sec:mass}

\begin{definition}[Pseudo-analytic mass]\label{def:mass}
The \emph{pseudo-analytic mass} of $D \in \BV(\Omega)$ is
\begin{equation}\label{eq:mass-def}
\Mcal(D) := \int_\Omega \Theta_D \;=\; \int_\Omega \frac{|\Bcal|^2}{1 - |\mu|^2}\, dx\, dy \;\in\; [0, \infty].
\end{equation}
\end{definition}

\begin{theorem}[Mass invariance]\label{thm:mass-invariance}
$\Mcal$ is an invariant of the gauge-diffeomorphism equivalence class: $D_1 \sim D_2 \implies \Mcal(D_1) = \Mcal(D_2)$.
\end{theorem}

\begin{proof}
Gauge invariance is immediate from Theorem~\ref{thm:main-invariance}(i). For diffeomorphism invariance, Theorem~\ref{thm:main-invariance}(ii) gives $\Theta_{\Phi^*D} = \Phi^*\Theta_D$, so the integral over $\Omega_1$ of the left equals the integral over $\Omega_2$ of $\Theta_D$ by the change-of-variables formula.
\end{proof}

\begin{proposition}[Vanishing characterization]\label{prop:mass-vanishing}
On any nonempty domain $\Omega$, $\Mcal(D) = 0$ if and only if $D$ is analytic ($\Bcal \equiv 0$).
\end{proposition}

\begin{proof}
The integrand $|\Bcal|^2/(1-|\mu|^2)$ is non-negative and continuous; it vanishes pointwise iff $\Bcal = 0$ pointwise. The integral vanishes iff the integrand vanishes a.e., which by continuity means everywhere.
\end{proof}

\begin{remark}[Regularity dependence]\label{rem:vanishing-regularity}
The proof of Proposition~\ref{prop:mass-vanishing} relies on the standing continuity hypothesis on the BV data introduced in Section~\ref{sec:gauge}. In a lower-regularity class (for instance, BV data in $L^p$), the conclusion would read $\Bcal = 0$ almost everywhere rather than pointwise.
\end{remark}

\begin{corollary}[Separation]\label{cor:separation}
If $\Mcal(D_1) \neq \Mcal(D_2)$, then $D_1 \not\sim D_2$.
\end{corollary}

The condition $\Mcal(D_1) = \Mcal(D_2)$ is necessary but not sufficient for equivalence; Example~\ref{ssec:ex-equal-mass} below exhibits two equations of equal mass that are nevertheless inequivalent.

\begin{remark}[Zero-locus and vanishing-order invariants]\label{rem:divisor}
The zero locus $\{\Bcal = 0\}$ is preserved by both equivalences (Proposition~\ref{prop:B-zero-invariant}); this is a set-theoretic invariant requiring only continuity of $\Bcal$ and the nowhere-vanishing of the gauge factor $\phi/\bar\phi$ and the pullback factor $J/K$. Multiplicity is a finer matter. In analytic or sufficiently regular settings where the local order of vanishing of $\Bcal$ at an isolated zero is well-defined and preserved under the allowed coordinate changes (for instance, when $\Bcal$ is real-analytic and one restricts to real-analytic diffeomorphisms, or when the multiplicity admits a $C^k$-stable formulation for the relevant $k$), one obtains a finer multiplicity invariant. We do not develop this further here; the mass is the focus of the present paper, and the set-theoretic zero locus suffices for the inequivalence example below.
\end{remark}

\subsection{Comparison across subregions}

If $\Omega'$ is an open subset of $\Omega$, then $\Mcal(D|_{\Omega'}) \leq \Mcal(D|_\Omega)$, with equality iff $\Bcal = 0$ a.e.\ on $\Omega \setminus \Omega'$. (When $\Omega \setminus \Omega'$ has empty interior, this condition is vacuous on the measure-theoretic level; when its interior is non-empty and $\Bcal$ is continuous, vanishing a.e.\ on the interior forces $\Bcal \equiv 0$ there.) The mass is monotone in the domain, so a positive mass on a subregion certifies non-vanishing $\Bcal$ on that subregion in a quantitative way.

\section{Examples}\label{sec:examples}

\subsection{Constant data on the disk}\label{ssec:ex-constant}

Let $\Omega = \D$ and consider the BV equation with constant data: $\mu(z) \equiv \mu_0 \in \D$, $\Bcal(z) \equiv \Bcal_0 \in \C$, and $\Acal, \Fcal$ arbitrary continuous (their values do not affect $\Theta$). Then
\[
\Theta = \frac{|\Bcal_0|^2}{1 - |\mu_0|^2}\, dx\, dy
\]
is a constant 2-form, and
\[
\Mcal(D) = \frac{|\Bcal_0|^2}{1 - |\mu_0|^2} \cdot \pi.
\]
For fixed $\mu_0$ and varying $|\Bcal_0|$, the mass takes any value in $[0, \infty)$.

\subsection{A continuous family of inequivalent equations}\label{ssec:ex-family}

Within the constant-data setting, fix $\mu_0 = 0$ and let $\Bcal_0 = t \in \R_{\geq 0}$. The BV equation
\[
D_t := (0, 0, t, 0) \quad \text{on } \D
\]
has $\Mcal(D_t) = \pi t^2$. For $t_1 \neq t_2$, $\Mcal(D_{t_1}) \neq \Mcal(D_{t_2})$, so $D_{t_1} \not\sim D_{t_2}$ by Corollary~\ref{cor:separation}. The family $\{D_t : t > 0\}$ thus consists of pairwise inequivalent pseudo-analytic equations on $\D$, parametrized by a real positive number.

In particular, the gauge-diffeomorphism quotient of $\BV(\D)$ has \emph{at least} the cardinality of the continuum within the pseudo-analytic class. The map $t \mapsto \Mcal(D_t) = \pi t^2$ is injective on the family $\{D_t : t \geq 0\}$ and continuous in the parameter, so $\Mcal$ provides an injective continuous coordinate distinguishing these equivalence classes.

\subsection{Equal mass does not imply equivalence}\label{ssec:ex-equal-mass}

The mass is a necessary but not sufficient invariant. Consider on $\D$:
\[
D_1 := (0,\, 0,\, \tfrac{1}{\sqrt{2}},\, 0), \qquad D_2 := (0,\, 0,\, z,\, 0).
\]
The first has constant $\Bcal_1 \equiv 1/\sqrt{2}$, the second has $\Bcal_2(z) = z$.

\emph{The masses agree.} For $D_1$, $\Mcal(D_1) = (1/2)\cdot \pi = \pi/2$. For $D_2$, working in polar coordinates,
\[
\Mcal(D_2) = \int_\D |z|^2\,dx\,dy = \int_0^{2\pi}\int_0^1 r^2 \cdot r\, dr\, d\theta = 2\pi \cdot \tfrac{1}{4} = \tfrac{\pi}{2}.
\]
Hence $\Mcal(D_1) = \Mcal(D_2) = \pi/2$.

\emph{The equations are inequivalent.} The function $\Bcal_1$ is nowhere zero on $\D$, while $\Bcal_2(z) = z$ has exactly one zero on $\D$, at the origin. Suppose for contradiction that $D_1 \sim D_2$. By Proposition~\ref{prop:B-zero-invariant}, the zero locus of $\Bcal$ is preserved (as a set, modulo the relabelling induced by the diffeomorphism) under the equivalence. But the zero locus of $\Bcal_1$ is empty while the zero locus of $\Bcal_2$ is $\{0\}$, a non-empty set. The gauge factor $\phi/\bar\phi$ is nowhere zero, so it cannot create or destroy zeros of $\Bcal$; a diffeomorphism $\Phi$ relabels the zero locus of $\Bcal$ via $\Phi$, which sends the empty set to the empty set bijectively, and therefore cannot account for the appearance of a zero. Contradiction.

Therefore $D_1 \not\sim D_2$ despite $\Mcal(D_1) = \Mcal(D_2)$. The mass is one invariant among several; the zero locus of $\Bcal$ is another, and they detect different aspects of the equivalence class. (No appeal to multiplicities is needed for this example: the cardinality of the zero set already separates the two equations.)

\section{Concluding remarks}\label{sec:concluding}

\subsection*{Summary}

We have shown that every smooth first-order real planar elliptic system can be written, by an algebraic and differential pipeline, as a Beltrami--Vekua equation \eqref{eq:BV-final} with data $(\mu, \Acal, \Bcal, \Fcal)$. On the resulting space, the joint action of multiplicative gauges $w \mapsto \phi w$ and orientation-preserving diffeomorphisms generates an equivalence relation $\sim$ that absorbs both classical reductions of Vekua's program. The 2-form $\Theta = |\Bcal|^2/(1-|\mu|^2)\,dx\,dy$ is intrinsic to $\sim$-classes (Theorem~\ref{thm:main-invariance}), and its total mass $\Mcal(D)$ is a $[0,\infty]$-valued invariant (Theorem~\ref{thm:mass-invariance}) vanishing exactly on the analytic class and separating a continuous family of pseudo-analytic equations (Section~\ref{ssec:ex-family}).

Two structural by-products supplement the main result. First, the classical two-stage reduction --- uniformize, then gauge --- requires only one variable-principal-part PDE solve in the BV framework: the same diffeomorphism that solves the Beltrami equation reduces the subsequent gauge step to a flat $\bar\partial$-equation (Proposition~\ref{prop:single-solve}). Second, the asymmetry between $\Bcal$ and $\Fcal$ in the gauge action --- multiplicative-by-phase versus multiplicative-by-arbitrary-scalar --- is the structural reason $|\Bcal|$ admits a moduli-theoretic 2-form invariant while $|\Fcal|$ does not at this algebraic level (Lemma~\ref{lem:no-F-invariant}).

\subsection*{Regularity scope}

The results of this paper are stated and proved in the smooth class of Definition~\ref{def:smooth-class}: principal-part coefficients in $C^1$, lower-order data and forcing in $C^0$. Section~\ref{sec:reduction} additionally requires $\mu \in C^{1,\alpha}$ and bounded domains for the Cauchy-transform step in the single-solve normal form. These are the classical $C^1$ subset of Vekua's Sobolev tier ($D_{m+1,p}$ on the principal part, $D_{m,p}$ on lower-order data; see \cite[\S 7.1]{vekua}), and the natural regularity for any diffeomorphism-based theory in this area. The algebraic portions of the construction admit a natural Sobolev formulation with pointwise identities interpreted a.e.; a full Sobolev development of the diffeomorphism action and the invariant 2-form (in particular, specifying representatives, almost-everywhere pullbacks, and the function spaces in which $\Theta$ lives) requires the standard quasiconformal apparatus and is not pursued here. Bojarski's complementary construction \cite{bojarski} works under measurability of all coefficients and produces a different complex form (with two principal-part terms); the bundle map between his form and the BV form exists --- both come from the same skeleton \eqref{eq:real-system} --- but its explicit coefficient transformation is left for a follow-up.

\subsection*{The lower-order coefficient \texorpdfstring{$\Acal$}{A}}

The proof of Lemma~\ref{lem:no-F-invariant} showed that gauge transformations can shift $\Acal$ by an arbitrary logarithmic derivative, so no pointwise algebraic invariant of $\Acal$ exists. The transformation law $\Acal \mapsto \Acal - \phi_{\bar z}/\phi + \mu\phi_z/\phi$ is, however, an additive shift by a deformed exterior derivative --- in the case $\mu = 0$, the standard gauge transformation law of a connection 1-form on a line bundle. Gauge invariants involving $\Acal$ should therefore be \emph{differential} (analogues of curvature), obtained by applying a deformed $\partial$-operator to $\Acal$. We have not pursued this construction here.

\subsection*{Connection to the variable-elliptic-structures obstruction}

The pipeline of Section~\ref{sec:BV-derivation}, building on the variable elliptic structures program \cite{ves}, derives the BV form via the intrinsic obstruction $G = G_1 + G_2 \ii$ measuring the failure of the moving generator $\ii \in \Az$ to be holomorphic on its own algebra. The obstruction $G$ and the mass $\Mcal$ live at different levels: $G$ is built from derivatives of the principal-part data $\alpha, \beta$ alone, while $\Mcal$ is built from the lower-order $\Bcal$ together with the principal-part modulus $|\mu|$. A BV equation can have $G \equiv 0$ (rigid structure) and arbitrarily large mass (Section~\ref{ssec:ex-constant} has constant $\mu$, hence $G \equiv 0$, but $\Mcal$ takes any value in $(0, \infty)$). The two invariants are independent.

\subsection*{Use of Generative AI Tools}
\medskip

The author discloses the use of Anthropic's Claude (Claude Opus 4.7,
accessed through the Claude.ai web interface during April--May 2026)
in the preparation of this manuscript. The tool was used as follows:

\begin{enumerate}
\item[(i)] \emph{Drafting and revision of expository prose.} The
introduction, remarks, and concluding section were developed in
iterative dialogue with the tool to improve clarity and organization.
All claims and their precise wording were reviewed by the author.

\item[(ii)] \emph{Stress-testing of definitions, hypotheses, and proof
arguments through dialogue.} Conversation-based review surfaced several
issues during writing, which the author then resolved. The mathematical
content of the corrections was determined by the author.

\item[(iii)] \emph{Verification of selected algebraic identities.}
The tool was prompted to expand and check expressions which the author
had derived independently; in every case the author independently
re-verified the expansion. The tool was not used to generate formulas,
derivations, or proofs.

\end{enumerate}

The tool was not used to generate research data, figures, tables,
mathematical results, proofs, or any content presented as the author's
own contribution. All mathematical statements, proofs, derivations, and
interpretations in this paper were developed and verified by the author,
who takes full responsibility for the accuracy, originality, and
integrity of all content.

\bibliographystyle{plain}

\end{document}